\documentclass[11pt]{amsart}

\usepackage{latexsym}
\usepackage{amssymb}
\usepackage{amsmath}
\usepackage{amscd}
\usepackage{mathrsfs}
\usepackage{url}

\newtheorem{theo}{Theorem}[section]

\newtheorem{prop}[theo]{Proposition}
\newtheorem{cor}[theo]{Corollary}

\theoremstyle{definition}

\newtheorem{rem}{Remark}[section]

\newtheorem{notn}{Notation}[section]

\DeclareMathOperator{\Vis}{Vis}

\DeclareMathOperator{\SL}{{\rm SL}}
\DeclareMathOperator{\SO}{{\rm SO}}

\DeclareMathOperator{\On}{O}

\newcommand{\nor}{{N}_G} 
\newcommand{\norone}{{N}^1_G}
\newcommand{\cen}{{Z}_G} 
\DeclareMathOperator{\supp}{{\rm supp}} 
\DeclareMathOperator{\Stab}{{\rm Stab}} 
\DeclareMathOperator{\Ad}{{\rm Ad}} 
\DeclareMathOperator{\diag}{\rm{diag}} 
\DeclareMathOperator{\Lie}{{\rm Lie}} 
\DeclareMathOperator{\Cc}{{\rm C}_{{\rm c}}} 
\DeclareMathOperator{\dd}{{\rm d}}

\renewcommand{\phi}{\varphi}

\newcommand{\vect}[1]{{\boldsymbol{#1}}}

\newcommand{\vv}{\vect{v}}
\newcommand{\vx}{\vect{x}}

\newcommand{\field}[1]{\mathbb{#1}} 
 
\newcommand{\R}{\field{R}}
\newcommand{\N}{\field{N}}

\newcommand{\HH}{\field{H}}

\newcommand{\Sn}{\field{S}}

\providecommand{\abs}[1]{\lvert#1\rvert}
\providecommand{\Abs}[1]{\left\lvert #1 \right\rvert}
\providecommand{\norm}[1]{\lVert#1\rVert}
\providecommand{\inpr}[2]{\langle #1,\, #2\rangle}

\providecommand{\trn}[1]{{\,^{\bf t}\!#1}}
\providecommand{\card}[1]{\#(#1)}

\newcommand{\innprod}{\inpr}

\newcommand{\inv}{^{-1}}

\newcommand{\cl}[1]{\overline{#1}}

\newcommand{\la}[1]{\mathfrak{\lowercase{#1}}}

\newcommand{\cA}{\mathcal{A}}

\newcommand{\cC}{\mathcal{C}}

\newcommand{\cD}{\mathcal{D}}

\newcommand{\cH}{\mathcal{H}}

\newcommand{\cI}{\mathcal{I}}

\newcommand{\cK}{\mathcal{K}}

\newcommand{\cO}{\mathcal{O}}

\newcommand{\cS}{\mathcal{S}}
\newcommand{\sS}{\mathscr{S}}

\newcommand{\cW}{\mathcal{W}}

\newcommand{\gmg}{{G/\Gamma}}

\newcommand{\pN}{p_N}
\newcommand{\nc}{\text{nc}}

\newcommand{\simm}[1]{\ \stackrel{#1}{\approx}\ }

\begin{document}

\title[Evolution of curves under geodesic flow]{Asymptotic evolution
  of smooth curves under geodesic flow on hyperbolic manifolds-II}

\author{Nimish A.  Shah} \address{Tata Institute of Fundamental
  Research, Mumbai 400005, INDIA} \email{nimish@math.tifr.res.in}


\begin{abstract}
  Extending the earlier results for analytic curve segments, in this
  article we describe the asymptotic behaviour of evolution of a
  finite segment of a $C^n$-smooth curve under the geodesic flow on
  the unit tangent bundle of a finite volume hyperbolic
  $n$-manifold. In particular, we show that if the curve satisfies
  certain natural geometric conditions, the pushforward of the
  parameter measure on the curve under the geodesic flow converges to
  the normalized canonical Riemannian measure on the tangent bundle in
  the limit. We also study the limits of geodesic evolution of
  shrinking segments.

  We use Ratner's classification of ergodic invariant measures for
  unipotent flows on homogeneous spaces of $\SO(n,1)$, and an
  observation relating local growth properties of smooth curves and
  dynamics of linear $\SL(2,\R)$-actions.
  
\end{abstract}

\subjclass{37A17 (Primary); 22E40, 37D40 (Secondary)}

\keywords{Equidistribution, geodesic flow, hyperbolic manifold,
  shrinking curves, unipotent flow, Ratner's theorem}

\maketitle

\section{Introduction}
Let $M$ be a hyperbolic $n$-dimensional manifold of finite volume,
$p:T^1(M)\to M$ be the unit tangent bundle over $M$, and
$\{g_t\}_{t\in\R}$ denote the geodesic flow on $T^1(M)$. Let
$\pi:\HH^n\to M$ be a locally isometric universal cover of $M$ and
$D\pi:T^1(\HH^n)\to T^1(M)$ the corresponding covering map. If
$\{\tilde g_t\}$ denotes the geodesic flow on $T^1(\HH^n)$, then
$p(\tilde g_t(v))\xrightarrow{t\to\infty} \Vis(v)$ for all $v\in
T^1(\HH^n)$, where $\Vis:T^1(\HH^n)\to \partial \HH^n\cong \Sn^{n-1}$
denotes the visual map. We define
\begin{equation}
  \label{eq:95}
  \begin{split}
    \sS=\{\partial\HH^m\subset\Sn^{n-1}&:\text{$\HH^m\hookrightarrow
      \HH^n$
      is an isometric embedding such that} \\
    &\text{$\pi(\HH^m)$ is closed in $M$, where $2\leq k\leq n-1$}\}.
  \end{split}
\end{equation}
Then $\sS$ is a countable collection of proper closed subspheres of
$\Sn^{n-1}$ (\cite{Shah:tot-geod},\cite[\S(5.2)]{Shah:uniform}).

Let $I$ be a compact interval with nonempty interior. Let $\psi:I\to
T^1(M)$ be an continuous map with the following property: If
$\tilde\psi:I\to T^1(\HH^n)$ is any continuous lift of $\psi$ under
$D\pi$, then
\begin{enumerate}
\item[a)] $\Vis\circ\tilde\psi\in C^n(I,\Sn^{n-1})$,
\item[b)] the first derivative $(\Vis\circ\tilde\psi)^{(1)}(s)\neq 0$
  for Lebesgue a.e.\ $s\in I$, and
\item[c)] for any $S\in\sS$, $\Vis(\tilde\psi(s))\not\in S$ for
  Lebesgue a.e.\ $s\in I$.
\end{enumerate}

\begin{theo}
  \label{thm:basic}
  Let the notation be as above. Then for any $f\in\Cc(T^1(M))$
  \begin{equation}
    \label{eq:37a}
    \frac{1}{\abs{I}}\int_{I}f(g_t\psi(s))\,ds\xrightarrow{t\to\infty}
    \int_{T^1(M)}f\,d\mu,
  \end{equation}
  where $\abs{I}$ denotes the Lebesgue measure of $I$ and $\mu$
  denotes the normalized measure associated to the canonical
  Riemannian volume form on $T^1(M)$.
\end{theo}

For the motivation for considering the question the reader is referred
to \cite{Shah:son1}, where the result was proved in the special case
of analytic curve segments $\psi$. The proof in \cite{Shah:son1}
involves the use of `$(C,\alpha)$-growth properties', in the sense of
Kleinbock and Margulis~\cite{Klein+Mar:Annals98}, of
finite-dimensional spaces of analytic functions. As these could not be
extended to smooth functions, the analogous result could not be proved
by the techniques in \cite{Shah:son1} for smooth curve segments,
though the conclusion could be expected to hold in that generality, as
was especially commented to the author by Peter Sarnak in response to
the result in \cite{Shah:son1}.

In this article we overcome this difficulty by making a new
observation of a linear dynamical nature. It implies that if we
approximate an arbitrarily short piece of a $C^n$-curve by a
polynomial curve degree at most $n$, then the geodesic flow expands
both the approximating curves into long curves of roughly fixed
lengths while still keeping them sufficiently close. This observation
allows us to use the growth properties of polynomial curves of bounded
degrees for linearization
method~\cite{Dani:rk=1,Shah:uniform,Dani+Mar:limit,Mozes+Shah:limit}.

On the space $C^n(I,T^1(M))$, we consider the topology of uniform
convergence up to $n$-derivatives. We now state a more robust form of
Theorem~\ref{thm:basic}.

\begin{theo}
  \label{thm:M}
  Let the map $\psi$ be as above. Then given $f\in \Cc(T^1(M))$ and
  $\epsilon>0$ there exists a neighbourhood $\Omega$ of $\psi$ in
  $C^n(I,T^1(M))$ and $T>0$ such that
  \begin{equation}
    \label{eq:37}
    \Abs{\frac{1}{\abs{I}}\int_{I}f(g_{t}\psi_1(s))\,ds -
      \int_{T^1(M)}f\,d\mu}<\epsilon, \quad \forall \psi_1\in\Omega,\
    \forall t>T.
  \end{equation}
\end{theo}

It may be noted that even for an analytic map $\psi$, the above
uniform version could not be proved using the methods of
\cite{Shah:son1}, because we do not have the $(C,\alpha)$-growth
property for linear span a neighbourhood of an analytic function.

\subsection{Evolution of general $C^n$-curves}
\label{subsec:geom}
Let $\bar\sS$ denote the collection of all closed totally geodesic
immersed submanifolds of $M$ (including $M$ itself). Given
$M_1\in\bar\sS$, let $\sS(M_1)\subset \sS\cup \{\Sn^{n-1}\}$ be the
collection of the boundaries of all possible lifts of $M_1$ in
$\HH^n$. Let $\mu_{M_1}$ denote the probability measure which is the
normalized measure corresponding to the canonical Riemannian volume
form on $T^1(M_1)\subset T^1(M)$. Given any $S\in\sS$ or
$S=\Sn^{n-1}$, define
\begin{equation}
  \label{eq:106}
  S^\ast=S\setminus \bigcup_{\substack{S_1\subset S,\ \dim S_1<\dim S\\S_1\in \sS}} S_1.
\end{equation}

Let $\psi\in C^n(I,T^1(M))$. Let $\tilde\psi\in C^n(I,\HH^n)$ denote a
lift of $\psi$ under $D\pi$. We define
\begin{align}
  \label{eq:25}
  I(S)&=\{s\in I: \Vis(\tilde\psi(s))\in S^\ast\},\\
  I(M_1)&=\cup_{S\in\sS(M_1)} I(S).
\end{align}
In particular,
\[
I(M)=I(\Sn^{n-1})=\{s\in I:\Vis(\tilde\psi(s))\not\in S,\ S\in\sS\}.
\]

\begin{theo}
  \label{thm:smooth-general}
  Suppose that $(\Vis\circ\tilde\psi)^{(1)}(s)\neq 0$ for almost all
  $s\in I$. Then given any $f\in\Cc(T^1(M))$,
  \begin{equation}
    \label{eq:29}
    \lim_{t\to\infty} \int_If(a_{t}\psi(s))\,ds= \sum_{M_1\in \bar\sS}
    \abs{I(M_1)}\int_{T^1(M)}f\,d\mu_{M_1}.
  \end{equation}
\end{theo}

The above statement is obtained as a consequence of results about
limiting distributions of the evolution of shrinking curves under the
geodesic flow (see \S\ref{sec:short-geom}).

\subsection{Flows on homogeneous spaces}

The above results will be derived from their analogues in terms of
dynamics of flows on homogeneous space of Lie groups.

Let $G=\SO(n,1)=\SO(Q_n)$, where $Q_n$ is a quadratic form in $(n+1)$
real variables, defined as
\begin{equation}
  \label{eq:97}
  Q_n(y,x_1,\dots,x_{n-1},z)=2yz-(x_1^2+\dots+x_{n-1}^2).  
\end{equation}
Let $\Gamma$ a lattice in $G$. For $t\in\R$ and
$\vx=(x_1,\dots,x_{n-1})\in\R^{n-1}$, we define
\begin{equation}
  \label{eq:9}
  a_t=\left[\begin{smallmatrix} e^t \\ & 1 \\ && \ddots \\ &&& 1\\ &&&&
      e^{-t}
    \end{smallmatrix}
  \right]\in G \quad \text{and} \quad
  u(\vx)=
  \left[
    \begin{smallmatrix} 
      1 & x_1 &\dots& x_{n-1} &\norm{\vx}^2/2\\
      &   1  &     &        & x_1 \\
      &      &\ddots&        &\vdots\\
      &       &     & 1       & x_{n-1} \\
      & & & & 1
    \end{smallmatrix}
  \right]\in G.
\end{equation}

\begin{theo}
  \label{thm:main}
  Let $I$ be a compact interval with nonempty interior and $\phi:I\to
  \R^{n-1}$ be a $C^n$-map such that $\phi^{(1)}(s)\neq 0$ for all
  $s\in I$, and for any sphere or a proper affine subspace $S$ in
  $\R^{n-1}$,
  \begin{equation}
    \label{eq:74}
    \abs{\{s\in I: \phi(s)\in S\}}=0.
  \end{equation}
  Let $\phi_k\xrightarrow{k\to\infty}\phi$ be a convergent sequence in
  $C^n(I,\R^{n-1})$, $x_k\to x_0$ a convergent sequence in $G/\Gamma$
  and $t_k\to\infty$ in $\R$.  Then for any $f\in\Cc(\gmg)$,
  \begin{equation}
    \label{eq:4}
    \lim_{k\to\infty}\frac{1}{\abs{I}}\int_I
    f(a_{t_k}u(\phi_k(s))x_k)\,ds=\int_{\gmg} f\,d\mu_G, 
  \end{equation}
  where $\mu_G$ is the unique $G$-invariant probability measure on
  $\gmg$.
\end{theo}

In fact, we shall obtain the following more general version. Let
$P^-=\{g\in G: \cl{\{a_tga_t\inv:t>0\}} \text{ is compact}\}$. Then
$P^-$ is a proper parabolic subgroup of $G$ and $P^-\backslash G$
naturally identifies with $\SO(n-1)\backslash\SO(n)\cong
\Sn^{n-1}$. Let $\cI:G\to P^-\backslash G\cong \Sn^{n-1}$ be the
corresponding map. We note that under this identification $G$ acts on
$\Sn^{n-1}$ by conformal transformations. For $m\geq 2$ and $g\in G$,
$\cI(\SO(m,1)g)$ is a subsphere of $\Sn^{n-1}$ of dimension $m-1$. We
note that $\cI(\nor(\SO(m,1)))=\cI(\SO(m,1))$. Let
\begin{equation}
  \label{eq:88}
  \sS=\{\cI(\SO(m,1)g):\text{$\nor(\SO(m,1))g\Gamma$ is
    closed},\,2\leq m\leq n-1,\,g\in G\};      
\end{equation}
Then $\sS$ is a countable collection of proper subspheres of
$\Sn^{n-1}$ (\cite[\S(5.2)]{Shah:uniform},\cite[Cor.A]{R:uniform}).

\begin{theo}
  \label{thm:curve} Let $\psi\in C(I,G)$ be such that $\cI\circ\psi\in
  C^n(I,\Sn^{n-1})$ and for any $S\in\sS$,
  \begin{equation}
    \label{eq:90}
    (\cI\circ\psi)^{(1)}(s)\neq 0 \quad\text{and}\quad   \cI\circ\psi(s)\not\in
    S \quad \text{for Lebesgue a.e.\ $s\in I$.}
  \end{equation}
  Then given a sequence $\{\psi_k\}_{k\in\N}\subset C(I,G)$ such that
  \begin{equation}
    \label{eq:98}
    \cI\circ\psi_k\xrightarrow{k\to\infty} \cI\circ\psi_k, \quad
    \text{in $C^n(I,\Sn^{n-1})$},  
  \end{equation}
  and sequences $t_k\to\infty$ in $\R$ and $x_k\to x_0=e\Gamma$ in
  $G/\Gamma$,
  \begin{equation}
    \label{eq:89}
    \lim_{k\to\infty}\frac{1}{\abs{I}}\int_I
    f(a_{t_k}\psi_k(s)x_k)\,ds=\int_{\gmg} f\,d\mu_G,
    \quad \forall f\in\Cc(G/\Gamma). 
  \end{equation}
\end{theo}

From this statement we will derive the following uniform version. Let
$\ell_h:G\to G$ denote the left translation by $h\in G$.

\begin{theo}
  \label{thm:uniform}
  Let $\psi:I\to G$ be a $C^n$-map such that if $h\in G$ and $S$ is a
  proper subsphere of $\Sn^{n-1}$ then for Lebesgue a.e.\ $s\in I$,
  \begin{equation}
    \label{eq:94}
    (\cI\circ\ell_h\circ\psi)^{(1)}(s)\neq 0 \quad\text{and}\quad 
    (\cI\circ\ell_h\circ\psi)(s)\not\in S.
  \end{equation}
  Then given $f\in\Cc(G/\Gamma)$, a compact set $\cK\subset G/\Gamma$
  and $\epsilon>0$, there exists a neighbourhood $\Omega$ of $\psi$ in
  $C^n(I,G)$ and a compact set $\cC$ in $G$ such that
  \begin{equation}
    \label{eq:83}
    \Abs{\frac{1}{\abs{I}}\int_If(g\psi_1(s)x)\,ds-\int_I
      f\,d\mu_G}<\epsilon, \quad \forall \psi_1\in\Omega,\ x\in\cK,\text{
      and } g\in G\smallsetminus \cC.
  \end{equation}
\end{theo}

\subsubsection*{Acknowledgment} {\small The author would like to thank
  Elon Lindenstrauss for discussions which led to some of the ideas
  used in Proposition~\ref{prop:main}. Thanks are due to S.G. Dani for
  useful remarks on an earlier version of this article.}

\section{Linear dynamics and growth properties of functions}
\label{sec:linear-dynamics}

Let $V=\oplus_{d=1}^{\dim{\la{g}}} \wedge^d\la{g}$, and consider the
$\oplus_{d=1}^{\dim{\la{g}}}\wedge^d\Ad$ representation of $G$ on
$V$. We fix an inner product $\langle \cdot ,\cdot \rangle$ on $V$ and
let $\norm{\cdot}$ denote the associated norm.

For $\mu\in\R$, we define
\begin{equation}
  \label{eq:11}
  V_\mu=\{v\in V:a_tv=e^{\mu t} v,\ t\in\R\}.
\end{equation}

If $\{\vx_1,\dots, \vx_{\dim\la{g}}\}$ is a basis of $\la{g}$
consisting of eigen-vectors of $\{a_t\}$, then there is a basis of
$V_\mu$ consisting of elements of the form
$\vx_{i_1}\wedge\cdots\wedge \vx_{i_d}$. The eigenvalues of $a_t$ on
$\la{g}$ other than $1$ are: $e^{t}$ with multiplicity $n-1$, and
$e^{-t}$ with multiplicity $n-1$. Therefore for $t>0$ the smallest
eigenvalue of $a_t$ on $V$ is $e^{-(n-1)t}$ and the largest one is
$e^{(n-1)t}$.  Therefore
\begin{equation}
  \label{eq:12}
  V=\oplus_{\mu=-(n-1)}^{n-1} V_\mu.
\end{equation}
Let $q_\mu:V\to V_\mu$ be the projection associated to this
decomposition.

\begin{notn}
  \label{notn:phi}
  Let $\phi_k\to \phi$ be a convergent sequence in $C^n(I,\R^{n-1})$
  such that
  \begin{equation}
    \label{eq:43}
    \rho_0:=\inf_{s\in I}\norm{\phi^{(1)}(s)}>0.  
  \end{equation}

  Let $M=\cen(A)\cap \SO(n)$. Then $M\cong\On(n-1)$ and
  $\cen(A)=AM$. We define the action of any $z\in \cen(A)$ on
  $\R^{n-1}$ by the relation, $u(z\cdot\vv):=zu(\vv)z\inv$ for all
  $\vv\in\R^{n-1}$. Then $M$ acts on $\R^{n-1}$ via its identification
  with $\On(n-1)$, and $a_t\cdot\vv=e^t\vv$.
\end{notn}

\begin{prop}[Basic Lemma]
  \label{prop:main}
  Given $C>0$, there exists $R_0>0$ such that for any sequence
  $t_k\to\infty$ in $\R$ there exists $k_0\in\N$ such that for any
  $x\in I=[a,b]$ and $v\in V$, there exists an interval
  $[s_k,s'_k]\subset I$ containing $x$ such that for any $k\geq k_0$,
  the following conditions are satisfied:
  \begin{align}
    \label{eq:5}
    e^{t_k}(s'_k-s_k)^n&<C,\\
    \label{eq:5b}
    \norm{a_{t_k}u(\phi_k(s_k))v}&\geq \norm{v}/R_0, \quad \text{if
      $s_k>a$,}\\
    \norm{a_{t_k}u(\phi_k(s'_k))v}&\geq \norm{v}/R_0, \quad \text{if
      $s'_k<b$.}
  \end{align}
\end{prop}

\begin{proof}
  If for every $R_0>0$ the above conditions are not satisfied, then
  after passing to a subsequence, there exist sequences $t_k\to\infty$
  and $R_k\to\infty$ in $\R$, $v_k\to v_0$ in $V$ with $\norm{v_0}=1$,
  and $[r_k,r'_k]\subset I$ such that $r_k\to r_0$, $r'_k\to r_0$ and
  the following holds:
  \begin{align}
    \label{eq:3}
    \sup_{r_k\leq s\leq r_k'}\norm{a_{t_k}u(\phi_k(s))v_k}&\leq R_k\inv\\
    \label{eq:3b}
    e^{t_k}\delta_k^n&\geq C, \quad\text{where $\delta_k=r_k'-r_k$.}
  \end{align}
  
  For any $k\in\N$, let $w_k=\phi_k(r_k)v_k$ and
  \[
  \phi_{k,r_k}(s):=\phi_k(r_k+s)-\phi_k(r_k), \quad\forall
  s\in[a-r_k,b-r_k].
  \]
  Then
  \begin{align}
    \label{eq:8}
    \sup_{s\in[0,\delta_k]}\norm{a_{t_k}u(\phi_{k,r_k}(s))w_k} \leq
    R_k\inv.
  \end{align}
  Therefore, for any $0\leq \mu\leq n-1$,
  \begin{equation}
    \label{eq:13}
    \sup_{s\in[0,\delta_k]}\norm{q_{\mu}(u(\phi_{k,r_k}(s))w_k)}\leq
    R_k\inv e^{-\mu t_k}.
  \end{equation}
  Then by \eqref{eq:3b} we get
  \begin{equation}
    \label{eq:14}
    \sup_{s\in [0,\delta_k]}\norm{q_{\mu}(u(\phi_{k,r_k}(s))w_k)}\leq
    R_k\inv C\inv \delta_k^{n\mu}.
  \end{equation}

  Putting $\mu=1$ in \eqref{eq:14}, for any $v\in V_1$ with
  $\norm{v}=1$,
  \begin{equation}
    \label{eq:15}
    \sup_{s\in[0,\delta_k]}
    \abs{\innprod{u(\phi_{k,r_k}(s)w_k)}{v}}\leq R_k\inv
    C\inv\delta_k^{n}. 
  \end{equation}

  We define
  \[
  \phi_{0,r_0}(s)=\phi(r_0+s)-\phi(s),\quad \forall s\in[a-r_0,b-r_0].
  \]

  As $k\to\infty$, we have $R_k\inv\to 0$, $w_k\to
  w_0=u(\phi(r_0))v_0$, and $\delta_k\to 0$. Therefore by
  \eqref{eq:15},
  \begin{equation}
    \label{eq:24}
    q_{\mu}(u(\phi_{0,r_0}(0))w_0)=q_\mu(w_0)=0,\quad\forall\, 0\leq \mu\leq n-1.
  \end{equation}

  To derive estimate on higher derivatives from \eqref{eq:15} we will
  use the following elementary observation: If $\psi\in
  C^m([0,\delta],\R)$, then there exists $\xi\in (0,\delta)$ such that
  \begin{equation}
    \label{eq:96}
    \abs{\psi^{(m)}(\xi)}\leq 2^m3^{m(m-1)/2}\delta^{-m}\sup_{s\in
      I}\abs{\psi(s)}.
  \end{equation}
  
  To prove this by induction on $m$, we assume that there exists
  $\xi_1\in(0,\delta/3)$, and $\xi_2\in(2\delta/3,\delta)$ such that
  \[
  \abs{\psi^{(m-1)}(\xi_i)}\leq
  2^{m-1}3^{(m-1)(m-2)/2}(\delta/3)^{-(m-1)}\sup_{s\in
    I}\abs{\psi(s)}.
  \]
  Then by Rolls theorem, there exists $\xi\in(\xi_1,\xi_2)$ such that
  \[
  \abs{\psi^{(m)}(\xi)}\leq
  \abs{\psi^{(m-1)}(\xi_2)-\psi^{(m-1)}(\xi_1)}(\xi_2-\xi_1)\inv.
  \]
  Now \eqref{eq:96} follows, because $\abs{\xi_2-\xi_1}\geq \delta/3$.

  Combining \eqref{eq:15} with the above observation, for each $1\leq
  m\leq n$, and each $k$ and there exists $\xi_m(k)\in(0,\delta_k)$
  such that
  \begin{equation}
    \label{eq:16}
    \innprod{u(\phi^{(m)}_{k,r_k}(\xi_m(k)))w_k}{v}
    \leq (2^m3^{m(m-1)/2}C\inv)\delta_k^{n-m}R_k\inv. 
  \end{equation}

  There exists $a_0<b_0$ such that $0\in
  [a_0,b_0]\subset[a-r_0,b-r_0]$.  Then $[a_0,b_0]\subset
  [a-r_k,b-r_k']$ for all but finitely many $k$. As $k\to\infty$, we
  have $R_k\to\infty$, $\xi_m(k)\to 0$, and $\phi_{k,r_k}\to
  \phi_{0,r_0}$ in $C^n([a_0,b_0],\R^{n-1})$. Therefore by
  \eqref{eq:24} and \eqref{eq:16},
  \begin{equation}
    \label{eq:17}
    \innprod{u(\phi^{m}_{0,r_0}(0))w_0}{v}=0, \quad\forall v\in V_1,\
    \forall\,0\leq m\leq n.
  \end{equation}
  Hence due to Taylor's expansion,
  \begin{equation}
    \label{eq:27}
    \lim_{s\to 0}\norm{q_1(u(\phi_{0,r_0}(s))w_0)}/s^{n}=0.
  \end{equation}

  Next we will show that \eqref{eq:27} leads to a contradiction, using
  finite dimensional representations of $\SL(2,\R)$.

  In view of \eqref{eq:97}, let
  $H=\SO(Q_2)=\SO(2,1)\hookrightarrow\SO(n,1)$. Then $H$ is generated
  by $\{u(se_1)\}_{s\in\R}$, $A=\{a_t\}_{t\in\R}$ and
  $\{\trn{u(te_1)}\}_{t\in\R}$, where
  $e_1=(1,0,\dots,0)\in\R^{n-1}$. We realize $H$ as the image of
  $\SL(2,\R)$ under the Adjoint representation on its Lie algebra
  $\la{SL}(2,\R)$ such that $\diag(e^t,e^{-t})\in\SL(2,\R)$ maps to
  $a_{2t}\in H$.

  Let $\cW$ be a finite collection of irreducible $H$-submodules of
  $V$ such that
  \begin{equation}
    \label{eq:10}
    V=\oplus_{W\in\cW} W.
  \end{equation}
  For any $W\in \cW$, let $P_W:V\to W$ denote the projection with
  respect to the decomposition \eqref{eq:10}. In view of
  Notation~\ref{notn:phi}, for any $s\in [a-r_0,b-r_0]$, there exists
  $\theta(s)\in M\subset \On(n-1)$ such that
  \begin{equation}
    \label{eq:18}
    \theta(s)\cdot\phi_{0,r_0}(s)=\norm{\phi_{0,r_0}(s)}e_1.
  \end{equation}
  Let
  \begin{equation}
    \label{eq:21}
    \mu_0=\max\{\mu:q_{\mu}(w_0)\neq 0\}.
  \end{equation} 
  Then $\mu_0=\max\{\mu:q_{\mu}(zw_0)\neq 0\}$ for any $z\in M$. Let
  $W_0\in\cW$ be such that
  \begin{equation}
    \label{eq:19}
    P_{W_0}(q_{\mu_0}(\theta(0)\cdot w_0))\neq 0.  
  \end{equation}
  Therefore there exist $a_1<b_1$ such that $0\in[a_1,b_1]\subset
  [a-r_0,b-r_0]$ and
  \begin{align}
    \label{eq:20}
    \eta_0:=\inf_{s\in[a_1,b_1]}
    \norm{P_{W_0}(q_{\mu_0}(\theta(s)\cdot w_0))}>0.
  \end{align}
  By \eqref{eq:12} and \eqref{eq:24} we have that $-1\geq \mu_0\geq
  -(n-1)$. Recall that $\phi_{0,r_0}(0)=0$ and by \eqref{eq:43},
  $\abs{\phi^{(1)}(r)}\geq \rho_0$ for all $r\in I$. Let $s\in
  [a_1,b_1]$ and
  \begin{equation}
    \label{eq:28}
    h=\norm{\phi_{0,r_0}(s)}=\norm{\phi^{(1)}(r)}s\geq \rho_0 s, \quad\text{for some
      $r\in [a_1,b_1]$}.
  \end{equation}
  Then
  \begin{equation}
    \label{eq:99}
    \begin{split}
      \theta(s)q_1(u(\phi_{0,r_0}(s))w_0)&=q_1(\theta(s)u(\phi_{0,r_0}(s))w_0)\\
      &=q_1(u(he_1)\theta(s)w_0).
    \end{split}
  \end{equation}
  Now by the standard description of an irreducible representation of
  $\SL(2,\R)$, we have
  \begin{equation}
    \label{eq:23}
    \begin{split}
      &P_{W_0}(q_1(u(he_1)\theta(s)w_0))=q_1(u(he_1)P_{W_0}(\theta(s)w_0))\\
      &=h^{1-\mu_0}q_{\mu_0}(P_{W_0}(\theta(s)w_0)) +
      \sum_{\mu\leq\mu_0-1} h^{1-\mu} q_\mu(P_{W_0}(\theta(s)w_0)).
    \end{split}
  \end{equation}
  Since $P_{W_0}$ is norm decreasing and $\theta(s)\in \On(n-1)$, by
  \eqref{eq:20} and \eqref{eq:28}, we conclude that
  \begin{equation}
    \label{eq:26}
    \lim_{s\to 0} 
    \norm{q_1(u(\phi_{0,r_0}(s))w_0)}/s^{1-\mu_0}\geq\eta_0\rho_0^{1-\mu_0}>0.
  \end{equation}
  Since $0<1-\mu_0\leq n$, this contradicts \eqref{eq:27}.
\end{proof}

\begin{notn}
  \label{notn:phi-poly}
  For any $x\in I$, we define
  \begin{equation}
    \label{eq:51}
    P_{k,x}(s)=\phi_k(x)+\phi_k^{(1)}(x)s+\dots+\phi_k^{(n)}(x)s^n,
    \quad \forall s\in\R.
  \end{equation}
\end{notn}

\begin{cor}
  \label{cor:phi-poly}
  Let $R_0>0$ be as in Proposition~\ref{prop:main} for $C=1$. Then
  given a sequence $t_k\to\infty$ and $c>0$ there exists $k_1\in\N$
  such that for any $k\geq k_1$ and $x\in I=[a,b]$ there exist
  $s_k,s'_k\in I$ with $x\in[s_k,s'_k]$ such that for any $v\in V$, we
  have
  \begin{align}
    \label{eq:47a}
    e^{t_k}(s'_k-s_k)^n&<1\\
    \label{eq:47}
    \sup_{s_k\leq s\leq s_k'}
    \norm{a_{t_k}u(\phi_{k}(s))v-a_{t_k}u(P_{k,x}(s))v}&\leq
    c\norm{a_{t_k}u(\phi(x))v}\\
    \label{eq:47b}
    \norm{a_{t_k}u(\phi_{k}(s_k))v}&\geq \norm{v}/R_0, \quad\text{if
      $s_k>a$}\\
    \label{eq:47c}
    \norm{a_{t_k}u(\phi_{k}(s'_k))v}&\geq \norm{v}/R_0, \quad \text{if
      $s'_k<b$.}
  \end{align}
\end{cor}

\begin{proof}
  Given a sequence $t_k\to\infty$, let $k_0\in\N$ be as in
  Proposition~\ref{prop:main} for the above choices of $C=1$ and
  $R_0>0$. Let $x\in I$. By Proposition~\ref{prop:main} for any $k\geq
  k_0$ there exists a subinterval $J_k=[s_k,s'_k]$ containing $x$ such
  that
  \begin{equation}
    \label{eq:100}
    \abs{J_k}^n\leq e^{-{t_k}},    
  \end{equation}
  and \eqref{eq:47b} and \eqref{eq:47c} hold for all $v\in V$.
  
  Let $\delta>0$ be such that
  \begin{equation}
    \label{eq:52}
    \norm{u(y_1)-u(y_2)}_V\leq c,\quad \forall \abs{y_1-y_2}\leq \delta,\ y_1,y_2\in \R^{n-1},
  \end{equation}
  where $\norm{\cdot}_V$ denotes the operator norm.

  By \eqref{eq:100} and equi-continuity of the family
  $\{\phi_k^{(n)}\}$, there exists $k_1\geq k_0$ such that
  \begin{equation}
    \label{eq:54}
    \norm{\phi_k^{(n)}(x_1)-\phi_k^{(n)}(x_2)}\leq \delta, \quad \forall
    x_1,x_2\in J_k, \ \forall k\geq k_1.
  \end{equation}
  Therefore by Taylor's formula,
  \begin{equation}
    \label{eq:53}
    \abs{\phi_k(s)-P_{k,x}(s)}\leq \delta \abs{J_k}^n\leq \delta e^{-t_k},
    \quad \forall s\in J_k,\ \forall k\geq k_1.
  \end{equation}
  Let $k\geq k_1$ and $w_k=a_{t_k}u(\phi_k(x))v$. Then due to
  \eqref{eq:52} and \eqref{eq:53}, for all $s\in J_k$,
  \begin{equation}
    \begin{split}
      &\norm{a_{t_k}u(\phi_k(s))v - a_{t_k}u(P_{k,x}(s))v}\\
      &\leq
      \norm{u(e^{t_k}(\phi_k(s)-\phi_k(x)))w_k-u(e^{t_k}(P_{k,x}(s)-\phi_k(x)))w_k}
      \leq \delta\norm{w_k},
    \end{split}
  \end{equation}
  note that
  $(e^{t_k}(\phi_k(s)-\phi_k(x)))-(e^{t_k}(P_{k,x}(s)-\phi_k(x)))=
  e^k(\phi_k(s)-P_{k,x}(s))$.
\end{proof}

\begin{notn}
  \label{notn:z(s)}
  Let $e_1=(1,0,\dots,0)\in\R^{n-1}$. Then by Notation~\ref{notn:phi}
  there exists a continuous map $z:I\to \cen(A)$ such that
  \begin{equation}
    z(s)\cdot \phi^{(1)}(s)=e_1,\ \forall s\in I.\quad \text{Let } 
    R_1:=\sup_{s\in I}\norm{z(s)}_V.
  \end{equation} 
\end{notn}

\begin{prop}
  \label{prop:CD}
  Let $R_0>0$ be as in Proposition~\ref{prop:main} for $C=1$. Let
  $\cA$ be a linear subspace of $V$ and $\cC$ be a compact subset of
  $\cA$. Then given $\epsilon>0$ there exists a compact set
  $\cD\subset\cA$ containing $\cC$ such that the following holds:
  Given any neighbourhood $\Phi$ of $\cD$ in $V$, there exist a
  neighbourhood $\Psi$ of $\cC$ and $k_2\in\N$ such that and for any
  $v\in V$ with
  \begin{equation}
    \label{eq:57}
    \norm{v}\geq R_0R_1(\sup_{w\in\Phi}\norm{w}),
  \end{equation} 
  a subinterval $J\subset I$, and any $k\geq k_2$ with
  $e^{-t_k}<\abs{J}^n$, we have
  \begin{equation}
    \label{eq:56}
    \begin{split}
      \abs{\{s\in J&:z(s)a_{t_k}u(\phi_k(s))v\in \Psi\}}\\
      &\leq\epsilon\abs{\{s\in J:z(s)a_{t_k}u(\phi_k(s))v\in\Phi\}}.
    \end{split}
  \end{equation}
\end{prop}

In a special case of the above proposition when $\cA=\{0\}$ and
$\cC=\{0\}$, we have $\cD=\{0\}$ and the neighbourhoods $\Phi$ and
$\Psi$ can be described in terms of radii of balls centered at $0$.

\begin{proof}
  There exists $n_1\in\N$ such that if $P:\R\to\R^{n-1}$ is a
  polynomial map of degree at most $n$ and $v\in V$ then $s\mapsto
  u(P(s))v$ is a polynomial map of degree at most $n_1$.

  As in \cite[Prop.~4.2]{Dani+Mar:limit}, there exists a compact set
  $\cD\subset \cA$ containing $\cC$ such that given an open
  neighbourhood $\Phi_1$ of $\cD$ in $V$ there exists an open
  neighbourhood $\Psi_1$ of $\cC$ in $V$ contained in $\Phi_1$ such
  that for any polynomial map $\zeta:\R\to V$ of degree at most $n_1$
  and any bounded interval $J\subset\R$, if $\zeta(J)\not\subset
  \Phi_1$ then
  \begin{equation}
    \label{eq:58}
    \abs{\{s\in J:\zeta(s)\in \Psi_1\}}\leq \epsilon\abs{\{s\in
      J:\zeta(s)\in \Phi_1\}}.
  \end{equation}

  Now given a bounded open neighbourhood $\Phi$ of $\cD$ in $V$, we
  choose an open neighbourhood $\Phi_1$ of $\cD$ in $\Phi$ such that
  $\cl{\Phi_1}\subset \Phi$. Then we obtain an open neighbourhood
  $\Psi_1$ of $\cC$ contained in $\Phi_1$ as above. Let $\Psi$ be an
  open neighbourhood of $\cC$ contained in $\Psi_1$ such that
  $\cl{\Psi}\subset\Psi_1$.

  Let $\delta>0$ be such that $2\delta$-tubular neighbourhoods of
  $\cl{\Phi_1}$ (respectively, $\cl{\Psi}$) is contained in $\Phi$
  (respectively, $\Psi_1$). Let $R=\sup_{w\in\Phi}\norm{w}$. For
  $c=\delta/(R_1R)>0$, let $k_1\in \N$ be as in the
  Corollary~\ref{cor:phi-poly}. Let $k_2\geq k_1$ be such that
  \begin{equation}
    \label{eq:60}
    \norm{z(r)-z(r')}_V\leq \delta/R,\quad \forall \abs{r'-r}\leq
    e^{-t_k},\ r,r'\in I,\ k\geq k_2.
  \end{equation}

  Let $J\subset I$ be an interval and $v\in V$ with $\norm{v}\geq
  R_0R_1R$.  Let $k\geq k_2$ be such that
  $e^{-t_{k}}<\abs{J}^n$. Define
  \begin{align}
    \label{eq:55}
    E&=\{s\in J:z(s)a_{t_k}u(\phi_k(s))v\in\Psi\}\\
    F&=\{s\in J:z(s)a_{t_k}u(\phi_k(s))v\in\Phi\}.
  \end{align}

  Suppose that $F_1$ be a connected component of $F$ intersecting
  $E$. Let $x\in F_1\cap E$. By Corollary~\ref{cor:phi-poly} there
  exists $J_k=[s_k,s'_k]\subset I$ containing $x$ such that
  \eqref{eq:47a} - \eqref{eq:47c} hold. Then by \eqref{eq:47b},
  \eqref{eq:47c}, definitions of $R_1$, $R$ and $\delta$, and
  \eqref{eq:60},
  \begin{align}
    \label{eq:62}
    F_1\cap E&\subset\{s\in J_k\cap F_1:
    z(s_k)a_{t_k}u(P_{k,x}(s))v\in \Psi_1\}
  \end{align}
  Since $e^{t_k}\abs{J_k}^n<1$ and $e^{t_k}\abs{J}^n>1$ and $x\in
  J\cap J_k$, we get $\{s_k,s_k'\}\cap
  J\smallsetminus\{a,b\}\neq\emptyset$.  Therefore due to \eqref{eq:47b}
  and \eqref{eq:47c}, $a_{t_k}u(\phi_k(J_k\cap F_1))v\not\subset
  \Phi$. Hence by \eqref{eq:47} and the choices of $c$ and $R_1$,
  \begin{equation}
    \label{eq:102}
    z(s_k)a_{t_k}u(P_{k,x}(J_k\cap F_1))v\not\subset\Phi_1. 
  \end{equation}
  Therefore, since $z\mapsto\zeta(s):=z(s_k)a_{t_k}u(P_{k,x}(s))v$ is
  a polynomial map of degree at most $n_1$, by \eqref{eq:58} applied
  to the interval $J_k\cap F_1$ in place of $J$ we deduce that
  \begin{equation}
    \label{eq:63}
    \abs{F_1\cap E}\leq \abs{\{s\in J_k\cap F_1: z(s_k)a_{t_k}u(P_{k,x}(s))v\in
      \Psi_1\}}\leq \epsilon\abs{F_1}.
  \end{equation}
  Since $F$ has at most countably many disjoint connected components
  intersecting $E$, like $F_1$ as above, from \eqref{eq:63} we
  conclude that $\abs{E}\leq\epsilon\abs{F}$.
\end{proof}

\subsection{Geometry of intersection with weakly stable subspace}

\begin{prop}
  \label{prop:Gp0}
  Let $H$ be a proper noncompact simple Lie subgroup of
  $G=\SO(n,1)$. Let $p_0\in \wedge^{\dim H}\la{H}\smallsetminus\{0\}$. Then
  $Gp_0$ is closed.
\end{prop}

\begin{proof}
  There may be a simple direct proof of this statement. Here we will
  quote from some earlier results.

  Note that $\Stab(p_0)=\norone(H)=M_1H$, where
  \[
  \norone(H)=\{g\in \nor(H):\det(\Ad g|_{\Lie(H)})=1\}
  \]
  and $M_1$ is the compact centralizer of $H$ in $G$. There exists
  $g\in G$ such that $gHg\inv=\SO(k,1)$ and $(gM_1g\inv)^0=\SO(n-k)$
  for some $2\leq k\leq n-1$. Now $\norone(\SO(k,1))$ is a symmetric
  subgroup of $G$ (see~\cite[pp.284--285]{EMS:counting}), and
  $\norone(\SO(k,1))$ stabilizes $gp_0$. Therefore by
  \cite[Corollary~4.7]{GOS:Satake}, the orbit $Gp_0=G(gp_0)$ is
  closed.
\end{proof}

In view of \eqref{eq:11}, we define
\begin{equation}
  \label{eq:101}
  V^-=\sum_{\mu<0} V_\mu,\quad V^0=V_0, \quad V^+=\sum_{\mu>0} V_\mu.  
\end{equation}
Then $V=V^-\oplus V^0\oplus V^+$.

\begin{prop}
  \label{prop:S}
  Let $p_0\in V$ be as in Proposition~\ref{prop:Gp0}. Let $g\in
  G$. Define
  \begin{equation}
    \label{eq:41}
    S=S_g=\{x\in\R^{n-1}: u(x)gp_0\in V^-+V^0\}.
  \end{equation}
  If $S\neq \emptyset$, then either $S$ is subsphere of a sphere in
  $\R^{n-1}$ or a proper affine subspace of $\R^{n-1}$.
\end{prop}

\begin{proof}  
  Since the orbit $Gp_0$ is closed, for every $x\in S$ there exists
  $\xi(x)\in G$ such that
  \begin{equation}
    \label{eq:42}
    a_tu(x)gp_0 \xrightarrow{t\to\infty} \xi(x)p_0
  \end{equation}
  and $\xi(x)p_0\in V^0$ is fixed by $A$. Let
  $F=\norone(H)=\Stab(p_0)$. Then the map $gF\mapsto gp_0$ from $G/F$
  to $V$ is a homeomorphism onto $Gp_0$. Therefore
  \begin{equation}
    \label{eq:44}
    a_tu(x)gF\xrightarrow{t\to\infty} \xi(x)F,
  \end{equation}
  and $A\subset \xi(x) F\xi(x)\inv$.

  Choose any $x_0\in S$. Let $p_1=\xi(x_0)p_0$,
  $H_1=\xi(x_0)H\xi(x_0)\inv$, and $F_1=\Stab(p_1)=\xi(x_0)
  F\xi(x_0)\inv$.  Then $A\subset F_1=\norone(H_1)$. Hence $A\subset
  H_1$.  As $\R$-rank of $G$ is one, there exists a Weyl group
  `element' $w\in H_1\subset F_1$ such that $w=w\inv$ and $waw\inv
  =a\inv$ for all $a\in A$. Now $G$ admits a Bruhat decomposition
  \begin{equation}
    \label{eq:46}
    G=P^- w U^- \cup P^-= P^-U^+\cup P^-w
  \end{equation}
  (see \cite[\S12.14]{Rag:book}), where
  \begin{equation}
    \label{eq:84}
    \begin{split}
      U^+&=\{h\in G: a_t\inv ha_t\xrightarrow{t\to\infty} e\} =
      \{u(\vx):\vx\in\R^{n-1}\}\\
      U^-&=\{h\in G: a_t ha_t\inv\xrightarrow{t\to\infty} e\} =
      \{\trn{u(\vx)}:\vx\in\R^{n-1}\},\\
      P^-&=\{h\in G: \cl{\{a_t ha_t\inv:t>0\}}\text{ is compact}\} =
      U^-\cen(A)=U^-AM.
    \end{split}
  \end{equation}

  Let $x\in S$. Put $g_1=g\xi(x_0)\inv$, and
  $\xi_1(x)=\xi(x)\xi(x_0)\inv$. Then \eqref{eq:44} is equivalent to
  \begin{equation}
    \label{eq:1}
    a_tu(x)g_1F_1\xrightarrow{t\to\infty} \xi_1(x)F_1.
  \end{equation}
  Since $w\in F_1$, by \eqref{eq:46} there exist $b\in P^-$ and $X\in
  \la{U}^+$ such that
  \begin{equation}
    \label{eq:92}
    u(x)g_1F_1=b\exp(X)F_1.
  \end{equation}
  Now
  \begin{equation}
    \label{eq:48}
    a_tu(x)g_1p_1=a_t(b\exp(X))p_1=(a_tba_t\inv)\exp(\Ad a_t(X))p_1. 
  \end{equation}
  Since $a_tba_t\inv\to b_0$ as $t\to\infty$ for some $b_0\in
  \cen(A)$, by \eqref{eq:42}
  \begin{equation}
    \label{eq:49}
    \exp(e^tX)p_1=\exp(\Ad a_t (X))p_1\xrightarrow{t\to\infty} b_0\inv\xi_1(x)p_1.
  \end{equation}
  Since $U^+$ is a unipotent group, the orbit $U^+p_1$ is a closed
  affine variety, and hence the map $h(U^+\cap F_1)\mapsto hp_1$ from
  $U^+/(U^+\cap F_1)\to \bar V$ is proper. Therefore from
  \eqref{eq:49} we conclude that $\exp(X)\in F_1$. Hence by
  \eqref{eq:92}, $u(x)g_1F_1=bF_1$. Therefore
  \begin{equation}
    \label{eq:50}
    S_g=\{x\in\R^{n-1}: u(x)\in P^-F_1g_1\inv, \ A\subset F_1\};
  \end{equation}
  we have proved the inclusion ``$\subset$'', and the converse holds
  because $p_1\in V^0$ and $P^-p_1\subset V^0+V^-$.
  
  Let $\cI:G\to P^-\backslash G\cong\Sn^{n-1}$ be the map as defined
  in the introduction. The right action of any $g\in G$ on
  $P^-\backslash G$ corresponds to a conformation transformation on
  $\Sn^{n-1}$.  Let $\cS:\R^{n-1}\to \Sn^{n-1}$ be the map defined by
  $\cS(x)=\cI(u(x))$ for all $x\in \R^{n-1}$. Then $\cS$ is the
  inverse stereographic projection. Since $A\subset F_1$, $P^-\cap
  F_1$ is a proper parabolic subgroup of $F_1$. In fact,
  $\cI(F_1)=\cI(\SO(k,1))$ for some $2\leq k\leq n-1$. Hence
  \[
  \cI(F_1)\cong (P^-\cap F_1) \backslash F_1\cong \Sn^{k-1}
  \]
  is a proper subsphere of $\Sn^{n-1}$. Therefore by \eqref{eq:50},
  $S_g$ is the inverse image of a proper subsphere of $\Sn^{n-1}$
  under the stereographic projection.
\end{proof}

\begin{rem}
  \label{rem:sS1}
  In the Proposition~\ref{prop:S}, suppose that
  $\pi(\norone(H))=\pi(F)$ is closed in $G/\Gamma$. Put
  $g=\gamma\in\Gamma$. If $x_0\in S_\gamma$, then
  $g_1=\gamma\xi(x_0)\inv$ and hence $F_1g_1\inv\Gamma =
  \xi(x_0)F\Gamma$ is closed. By \eqref{eq:88} and \eqref{eq:50},
  $S_\gamma=\cS\inv(S)$, where $S=\cI(F_1g_1\inv)\in\sS$.
\end{rem}

\section{Limiting measure and invariance under unipotent flow}
Let $\phi_k\to\phi$ be a convergent sequence in $C^n(I,\R^{n-1})$ as
in Notation~\ref{notn:phi}. Let $z:I\to \cen(A)$ be the continuous
function as in Notation~\ref{notn:z(s)} such that $z\cdot
\phi^{(1)}(s)=e_1$ for all $s\in I$. Let $g_k\to g_0$ be a convergent
sequence in $G$. Then $x_k=g_k\Gamma\to x_0=g_0\Gamma$ in $G/\Gamma$.

\begin{prop}
  \label{prop:nondiv}
  Given $\epsilon>0$ there exists a compact set $\cK\subset G/\Gamma$
  such that for any sequence $t_k\to\infty$,
  \begin{equation}
    \label{eq:67}
    \frac{1}{\abs{I}}\abs{\{s\in I:z(s)a_{t_k}u(\phi_k(s))x_i\in\cK\}}
    \geq 1-\epsilon,
    \quad \text{for all large $k\in\N$.}
  \end{equation}
\end{prop}

\begin{proof}
  Let $N$ be a maximal unipotent subgroup of $G$ such that $N/(N\cap
  \Gamma)$ is compact. Let $\la{N}$ denote the Lie algebra of $N$. Fix
  $\pN\in V\smallsetminus\{0\}$ such that $\pN\in
  \wedge^{\dim{\la{n}}}\la{n}$. Then $\Gamma\pN$ is discrete (see
  \cite{Dani:rk=1}). Let
  \begin{equation}
    \label{eq:65}
    0<r_3=\inf_{\substack{\gamma\in\Gamma}{k\in\N}}\norm{g_k\gamma\pN}.
  \end{equation}
  By Proposition~\ref{prop:CD} applied to $\cA=\{0\}$, given $0<R\leq
  r_3/R_0R_1$ there exists $r>0$ such that the following holds: Given
  any sequence $t_k\to\infty$ there exists $k_2\in\N$ such that for
  any interval $J\subset I$, $k\geq k_2$ with $e^{-t_k}\geq \abs{J}^n$
  and $\gamma\in\Gamma$,
  \begin{equation}
    \label{eq:66}
    \begin{split}
      &\abs{\{s\in J:\norm{z(s)a_{t_k}u(\phi_k(s))g_k\gamma\pN}< r\}}\\
      &\leq \epsilon\cdot \abs{\{s\in
        J:\norm{z(s)a_{t_k}u(\phi(s))g_k\gamma\pN}<R\}}.
    \end{split}
  \end{equation}

  By the proof of Dani's non-divergence criterion\cite{Dani:rk=1} for
  homogeneous spaces of rank one semisimple groups, the conclusion in
  the previous paragraph implies the existence of a compact set $\cK$
  such that \eqref{eq:67} holds; the choice of $\cK$ depends only on
  $r>0$ chosen above, not on the sequence $\{t_k\}$.
\end{proof}

Take a sequence $t_k\to\infty$ in $\R$.  Let $\lambda_k$ be the
probability measure on $G/\Gamma$ defined by
\begin{equation}
  \label{eq:mui} 
  \int_{\gmg} f\,{\dd}\lambda_k=\frac{1}{\abs{I}}\int_I
  f(z(s)a_{t_k}u(\phi_i(s))x_k)\,ds,\quad\forall f\in\Cc(\gmg).
\end{equation}

Then Proposition~\ref{prop:nondiv} implies the following:
  
\begin{theo}
  \label{thm:return}
  After passing to a subsequence
  $\lambda_k\stackrel{k\to\infty}{\to}\lambda$ in the space of
  probability measures on $G/\Gamma$ with respect to the weak-$\ast$
  topology.
\end{theo}

\begin{theo}
  \label{thm:invariant}
  The limit measure $\lambda$ is invariant under the action of $W$.
\end{theo}

\begin{proof}
  The proof follows from the same argument as in the Proof of
  \cite[Theorem~3.1]{Shah:son1}.
\end{proof}

The next result says that the limit measure is null on the parabolic
cylinders embedded in the cusps.

\begin{prop}
  \label{prop:cusps}
  Let $U$ be any maximal unipotent subgroup of $G$ containing $W$ and
  $x\in G$ such that $Ux$ is compact. Then $\lambda(\nor(U)x)=0$.
\end{prop}

\begin{proof}(cf.~\cite{R:uniform}) Since $G$ is a rank one group, $W$
  is contained in a unique maximal unipotent subgroup
  \cite[\S12.17]{Rag:book}. Therefore $U=U^+$ and
  $\nor(U)=\cen(A)U^+$. Let $C$ be any compact subset of
  $\nor(U^+)$. Then $C_1=\cl{\{a_{-t}Ca_t:t>0\}}$ is compact. Given
  $\epsilon>0$, let $\cK$ be as in Proposition~\ref{prop:nondiv}. Let
  $\cK_1=C_1\inv\cK$.

  Since $Ux$ is compact, there exists $u\in U\smallsetminus\{e\}$ such that
  $ux=x$. Then $a_{-t}ua_t\to e$ as $i\to\infty$. Therefore by
  \cite[\S1.12]{Rag:book}, $a_{-t}x\not\in\cK_1$ for all $t\geq T_0$
  for some $T_0>0$. Therefore $a_{-T_0}Cx\cap\cK=\emptyset$, or in
  other words, $Cx\cap a_{T_0}\cK=\emptyset$.

  Let $t_k'=t_k-T_0$. Then by Proposition~\ref{prop:nondiv}, for all
  large $k\in\N$,
  \begin{equation*}
    \label{eq:2}
    \abs{\{s\in I: a_{t_k'}u(\phi_k(s))x_k\in \cK\}}\geq (1-\epsilon)\abs{I},
  \end{equation*}
  and hence $\abs{\{s\in I: a_{t_k}u(\phi_k(s))x_k\in
    a_{T_0}\cK\}}\geq (1-\epsilon)\abs{I}$. Since $(G/\Gamma)\smallsetminus
  a_{T_0}\cK$ is a neighbourhood of $Cx$, we conclude that
  $\lambda(Cx)\leq \epsilon$.
\end{proof}

\section{Ratner's theorem and linearization method}

Our basic goal is to prove that the following result using Ratner's
description of the ergodic invariant measure for unipotent flows, and
the linearization technique in combination with the linear dynamical
results proved in \S\ref{sec:linear-dynamics}.

\begin{theo}
  \label{thm:G-inv}
  Let the measure $\lambda$ be as in
  Theorem~\ref{thm:invariant}. Suppose further that the limit function
  $\phi$ satisfies the following condition: For any $(n-2)$-sphere or
  a proper affine subspace $S_1$ contained in $R^{n-1}$,
  \begin{align}
    \label{eq:6}
    \abs{\{s\in I:\phi(s)\in S_1\}}=0.
  \end{align}
  Then measure $\lambda$ is $G$-invariant.
\end{theo}

The rest of this section is devoted to the proof of this theorem.

\subsection{Positive limit measure on singular sets}

Let $\cH$ be the collection of all closed connected subgroups $H$ of
$G$ such that $H\cap\Gamma$ is a lattice in $H$, and a nontrivial
unipotent one-parameter subgroup of $H$ acts ergodically on
$H/H\cap\Gamma$. Then $\cH$ is countable
(\cite[\S5.1]{Shah:uniform},\cite{R:measure}). For $H\in\cH$, we
define
\begin{align}
  \label{eq:59}
  N(H,W)&=\{g\in G:U\subset gHg\inv\},\\
  \label{eq:69}
  S(H,W)&=\bigcup_{\substack{{H'\subset H,\,\dim H'<\dim
        H}\\{H'\in\cH}}} N(H',W).
\end{align}
Then (see~\cite{Mozes+Shah:limit})
\begin{equation}
  \label{eq:35}
  N(H,W)\cap N(H,W)\gamma\subset S(H,W), \quad \forall \gamma\in
  \Gamma\smallsetminus \nor(H). 
\end{equation}

Suppose that $\lambda$ is not $G$-invariant. Then by Ratner's
theorem~\cite{R:measure}, since $\cH$ is countable, there exists
$H\in\cH$ such that $\dim H<\dim G$ and
\begin{equation}
  \label{eq:68}
  \lambda(\pi(N(H,W)))>0 \quad\text{and}\quad \lambda(\pi(S(H,W)))=0,
\end{equation}
where $\pi:G\to G/\Gamma$ is the natural quotient map.

\subsection{Algebraic consequence of accumulation of limit measure on
  singular sets}

Since $G$ is a semisimple group of real rank one and $H\in\cH$, if $H$
is not reductive then $H$ is contained in a unique maximal unipotent
subgroup intersecting $\Gamma$ in a cocompact lattice. Hence for any
$g\in N(H,W)$, we have that $gHg\inv$ is contained in a maximal
unipotent subgroup $U$ of $G$ containing $W$ such that $U\pi(g)$ is
compact, and $\pi(N(H,W))\subset \nor(U)\pi(g)$. Now by
Proposition~\ref{prop:cusps} we have $\pi(N(H,W))=0$. Thus in view of
\eqref{eq:68}, we conclude that $H$ is a reductive subgroup of $G$.

We choose a compact set $C\subset N(H,W)\smallsetminus S(H,W)\Gamma$ and
$\epsilon>0$ such that
\begin{equation}
  \label{eq:70}
  0<\epsilon<2\lambda(\pi(C)). 
\end{equation}

Let $H^\nc$ denote the subgroup of $H$ generated by all unipotent
one-parameter subgroups contained in it. Since $H$ is a proper
reductive subgroup of $G=\SO(n,1)$, we have that $H^\nc\cong \SO(k,1)$
for some $2\leq k\leq n-1$, and $H=Z_1H^\nc$ where $Z_1$ is a compact
central subgroup of $H$. Moreover $\nor(H^\nc)=M_1H^\nc$, where $M_1$
is the centralizer of $H^\nc$ in $G$ which is compact. Since
$H\in\cH$, we have that $\cl{H^\nc\Gamma}=H\Gamma$.

Let $\la{H}^\nc$ denote the Lie algebra associated to $H^\nc$ and
$\ell_0=\dim \la{H}^\nc$. Let
\begin{equation}
  \label{eq:30}
  p_0\in \wedge^{\ell_0}\la{H}^\nc\smallsetminus\{0\}.
\end{equation}
Then $F:=\Stab(p_0)=F$ and $\nor(H^\nc)p_0\subset\{p_0,-p_0\}$. If
$\gamma\in \nor(H^\nc)\cap\Gamma$, then
\[
\gamma H\Gamma=\gamma\cl{H^\nc\Gamma}=\cl{H^\nc\Gamma}=H\Gamma.
\]
Therefore $\gamma\in \nor(H)$. Thus
\begin{equation}
  \label{eq:31}
  \Gamma\cap \nor(H^\nc)=\Gamma\cap\nor(H).
\end{equation}

Let $X_0\in\Lie(W)$ and
\begin{equation}
  \label{eq:32}
  \cA=\{v\in \wedge^{\ell_0}\la{g}: v\wedge X_0=0\}.
\end{equation}
Then $\cA$ is a linear subspace of $V$. For any $g\in G$,
\begin{equation}
  \label{eq:33}
  gp_0\in\cA\Leftrightarrow \wedge^{\ell_0}\la{H}^\nc \wedge
  \Ad(g\inv)X_0=0\Leftrightarrow H^\nc \supset g\inv W g\Leftrightarrow
  g\in N(H,W). 
\end{equation}
Therefore for any $g\in G$, $g\in N(H,W)\Leftrightarrow gp_0\in \cA$.

Since $F/H$ is compact, and $H\Gamma$ is closed, we have that
$F\Gamma$ is closed. Therefore $\Gamma F$ is closed. By
Proposition~\ref{prop:Gp0}, the map $gF\mapsto gp_0$ from $G/F$ to $V$
is proper. Therefore $\Gamma p_0$ is closed in $V$. Hence $\Gamma p_0$
is discrete.

Given any compact set $\cD$ of $\cA$, we define
\begin{equation}
  \label{eq:34}
  \sS(\cD)=\{g\in G: gp_0,g\gamma p_0\in \cD \text{ for some } \gamma\in
  \Gamma\smallsetminus \nor(H^\nc)\}.
\end{equation}
Due to \eqref{eq:35} and \eqref{eq:31}, $\sS(\cD)\subset S(H,W)$ and
$\pi(\sS(\cD))$ is closed in $G/\Gamma$
\cite[Prop.~3.2]{Mozes+Shah:limit}. Now if $\cK$ is any compact set
contained in $G\smallsetminus \pi(\sS(\cD))$ then there exists a
neighbourhood $\Phi$ of $\cD$ in $V$ such that for any $g\in G$ and
$\gamma_1,\gamma_2\in \Gamma$,
\begin{equation}
  \label{eq:36}
  \pi(g)\in\cK,\, \{g\gamma_1p_0,g\gamma_2p_0\}\subset
  \cl{\Phi} \Rightarrow \gamma_1\in \gamma_2\nor(H^\nc) 
  \Rightarrow  g\gamma_1p_0=\pm g\gamma_2 p_0.   
\end{equation}

Let $\cC=C\cdot p_0\cup -(C\cdot p_0)\subset\cA$. Given $\epsilon>0$
we obtain $R_0$ and a compact set $\cD\subset \cA$ as in
Proposition~\ref{prop:CD}. Replacing $\cD$ by $\cD\cup -\cD$ we assume
that $\cD$ is symmetric about $0$. We choose a compact neighbourhood
$\cK$ of $\pi(C)$ contained in $G/\Gamma\smallsetminus\pi(\sS(\cD))$. We
take any symmetric neighbourhood $\Phi$ of $\cD$ such \eqref{eq:36}
holds. Then there exist a symmetric neighbourhood $\Psi$ of $\cC$ in
$V$ and $k_2\in\N$ such that given any subinterval $J$ of $I$ the
following holds: If $v\in V$ with
\begin{equation}
  \label{eq:38}
  \norm{v}\geq R_0R_1R, \quad \text{where $R=\sup_{w\in\Phi}\norm{w}$},  
\end{equation}
and if $k\geq k_2$ such that $e^{-t_k}<\abs{J}^n$, then
\begin{equation}
  \label{eq:56b}
  \begin{split}
    &\abs{\{s\in J:z(s)a_{t_k}u(\phi_k(s))v\in \Psi\}}\\
    \leq\epsilon\,&\abs{\{s\in J:z(s)a_{t_k}u(\phi_k(s))v\in\Phi\}}.
  \end{split}
\end{equation}

Let $\cO=\{\pi(g):gp_0\in\Psi, \pi(g)\in\cK\}$. Then $\cO$ is a
neighbourhood of $\pi(C)$. Take any $k\in\N$. Let
\begin{equation}
  \label{eq:72}
  E(k)=\{s\in I: z(s)a_{t_k}u(\phi_k(s))\pi(g_k)\in\cO\}.
\end{equation}
Then by \eqref{eq:70},
\begin{equation}
  \label{eq:61}
  \abs{E(k)}>2\epsilon\abs{I},\quad\text{$\forall$ large $k\in\N$}.
\end{equation}

Let $B$ denote the ball of radius $R_0R_1R$ centered at $0$. Let
\begin{equation}
  \label{eq:39}
  \Sigma_1=\Gamma p_0\cap B \text{ and } \Sigma_2=\Gamma
  p_0\smallsetminus B.
\end{equation}
For $j=1,2$, let
\begin{equation}
  \label{eq:40} 
  E_j(k)=\{s\in E(k):z(s)a_{t_k}u(\phi(s))g_k v\in\Psi, 
  \text{ for some } v\in \Sigma_j\}.
\end{equation}
Then $E(k)=E_1(k)\cup E_2(k)$.

Now using \eqref{eq:36} and \eqref{eq:56b}, by the argument as in the
proofs of \cite{Dani+Mar:limit}, \cite[Prop.~3.4]{Mozes+Shah:limit} or
\cite[Prop.~4.5]{Shah:son1}, there exists $k_4\in\N$ such that
\begin{equation}
  \label{eq:73}
  \abs{E_2(k)}\leq \epsilon \abs{I}.
\end{equation}

Next we want to prove that $\abs{E_1(k)}\leq \epsilon\abs{I}$ for all
large $k\in\N$. Suppose this is not true. Then
\begin{equation}
  \label{eq:79}
  \limsup_{k\in\N} \abs{E_1(k)}\geq \epsilon\abs{I}.
\end{equation}

\begin{prop}
  \label{prop:Sigma}
  There exists $v\in \Sigma_1$ such that
  \begin{equation}
    \label{eq:75}
    \abs{\{s\in I: u(\phi(s))g_0v\in V^-+V^0\}}\geq 
    \epsilon\abs{I}/\card{\Sigma_1}.
  \end{equation}
\end{prop}

\begin{proof}
  After passing to a subsequence, there exists $v\in \Sigma_1$ such
  that
  \begin{equation}
    \label{eq:78}
    \abs{\{s\in E(k): z(s)a_{t_k}u(\phi_k(s))g_kv\in\Psi\}}\geq
    \epsilon\abs{I}/\card{\Sigma_1}, \quad \text{$\forall$ large
      $k\in\N$}.   
  \end{equation}
  Let $q_+:V\to V^+$ be the projection associated to the decomposition
  $V=V^-\oplus V^0\oplus V^+$ as in \eqref{eq:101}. Let
  \begin{equation}
    \label{eq:76}
    E^\delta=\{s\in I: q_+(u(\phi(s))g_0v)\geq \delta\}.
  \end{equation}
  Since $\phi_k\to\phi$ uniformly on $I$, there exists $k_5\in\N$ such
  that if $k\geq k_5$ and if $s\in E^\delta$, then
  $q_+(z(s)u(\phi_k(s))g_kv)\geq \delta/2$. Hence
  \begin{equation}
    \label{eq:77}
    z(s)a_{t_k}u(\phi_k(s))g_kv=a_{t_k}(z(s)u(\phi(s))g_k)\not\in \Psi,
    \quad\text{$\forall$ large $k\in\N$}.   
  \end{equation}
  Therefore in $E^\delta\cap E_1(k)=\emptyset$ for all large
  $k\in\N$. Hence by \eqref{eq:78} and \eqref{eq:76},
  \[
  \abs{\{s\in I: u(\phi(s))g_0v\in V^-+V^0\}}=\lim_{\delta\to0}
  \abs{I\smallsetminus E^\delta}\geq\epsilon\abs{I}/\card{\Sigma_1}.
  \]
\end{proof}

Combining Proposition~\ref{prop:S} and Proposition~\ref{prop:Sigma},
there exists $\gamma\in\Gamma$ such that $\gamma p_0=v\in\Sigma_1$ and
\begin{equation}
  \label{eq:80}
  \abs{\{s\in I: \phi(s)\in S_{g_0\gamma}\}}\geq \epsilon\abs{I}/\card{\Sigma}.  
\end{equation}
This statement contradicts our assumption on $\phi$ as stated in
\eqref{eq:6}. Therefore \eqref{eq:79} does not hold, or in other
words,
\begin{equation}
  \label{eq:81}
  \limsup_{k\to\infty} \abs{E_1(k)}<\epsilon\abs{I}.
\end{equation}
Therefore by \eqref{eq:73},
\[
\abs{E(k)}\leq \abs{E_1(k)}+\abs{E_2(k)}< 2\epsilon\abs{I}, \quad
\text{$\forall$ large $k\in\N$}.
\]
This contradicts \eqref{eq:61}. Thus as noted above, due to Ratner's
theorem, $\lambda$ is $G$-invariant.  This completes the proof of
Theorem~\ref{thm:G-inv}.  \qed

\begin{rem}
  \label{rem:sS}
  If $g_0=e$, then in \eqref{eq:80}, $S_{g_0\gamma}=S_\gamma$ for some
  $\gamma\in\Gamma$, and by Remark~\ref{rem:sS1},
  $S_\gamma\in\cS\inv(\sS)$. Therefore if $g_0=e$, then the conclusion
  of Theorem~\ref{thm:G-inv} is valid if we assume the weaker
  condition on $\phi$ that \eqref{eq:6} holds for all
  $S_1=\cS\inv(S)$, where $S\in\sS$.
\end{rem}

\section{Deduction of the main results}

The following observation allows us to deduce the results stated in
the introduction from Theorem~\ref{thm:G-inv}.

\begin{prop}
  \label{prop:P-}
  Let $\{\theta_k\}$ and $\{\psi_k\}$ be uniformly convergent
  sequences of continuous maps from $I\to G$ such
  $P^-\theta_k(s)=P^-\psi_k(s)$ for all $s\in I$. Let $\{x_k\}$ be a
  sequence in $G/\Gamma$ and $t_k\to\infty$ be a sequence in
  $\R$. Suppose that there exists a probability measure $\mu$ on
  $G/\Gamma$ which is $\cen(A)$ invariant, and for any subinterval
  $J\subset I$ with nonempty interior and any $f\in\Cc(G/\Gamma)$ the
  following holds:
  \begin{equation}
    \label{eq:85}
    \lim_{k\to\infty}  \frac{1}{\abs{J}}\int_{J}f(a_{t_k}\theta_k(s)x_k)\,ds =
    \int_{G/\Gamma} f\,d\mu.
  \end{equation}
  Then for any $f\in\Cc(G/\Gamma)$,
  \begin{equation}
    \label{eq:85b}
    \lim_{k\to\infty}  \frac{1}{\abs{I}}\int_{I}f(a_{t_k}\psi_k(s)x_k)\,ds =
    \int_{G/\Gamma} f\,d\mu.
  \end{equation}
\end{prop}

\begin{proof}
  Since $P^-=U^-\cen(A)$ (see \eqref{eq:84}), for any $s\in I$ we
  express $\psi_k(s)=v(s)\zeta(s)\theta_k(s)$, where $\zeta_k(s)\in
  \cen(A)$ and $v(s)\in U^-$ are such that
  $\{s\mapsto\zeta_k(s)\}_{k\in\N}$ and $\{s\mapsto v_k(s)\}_{k\in\N}$
  are equi-continuous families of maps on $I$.

  Let $\epsilon>0$. Since $f$ is uniformly continuous on $G/\Gamma$
  and $\cl{\{v_k(s):s\in I, k\in\N\}}$ is compact in $U^-$, there
  exists $k_1\in\N$ such that for any $k\geq k_1$, $s\in I$ and $x\in
  G/\Gamma$, $\abs{f(a_{t_k}v_k(s)x)-f(a_{t_k}x)}\leq\epsilon$.

  Also there exists a finite partition of $I$ onto subintervals $J$'s
  such that if $s_1,s_2\in J$ and $x\in G/\Gamma$ then
  $\abs{f(\zeta_k(s_1)x)-f(\zeta_k(s_2)x)}<\epsilon$ for all $k\in\N$.

  If we fix some $s_J\in J$, then for all $s\in J$ and $k\geq k_1$,
  \begin{equation}
    \label{eq:82}
    \begin{split}
      &\abs{f(a_{t_k}\psi_k(s)x_k)-f(\zeta(s_J)a_{t_k}\theta_k(s)x_k)}\\
      &\leq
      \abs{f(a_{t_k}v_k(s)\zeta_k(s)\theta_k(s)x_k)-f(a_{t_k}\zeta_k(s)\theta_k(s)x_k)}\\
      &\qquad +
      \abs{f(\zeta_k(s)a_{t_k}\theta_k(s)x_k)-f(\zeta_k(s_J)a_{t_k}\theta_k(s)x_k)}
      \leq 2\epsilon.
    \end{split}
  \end{equation}
  Since $\int f(\zeta(s_J)y)\,d\mu(y)=\int f(y)\,d\mu$, from
  \eqref{eq:82} and \eqref{eq:85},
  \begin{equation}
    \label{eq:86}
    \Abs{\int_{J}f(a_{t_k}\psi_k(s)x_k)\,ds - \abs{J}\int_{G/\Gamma}
      f\,d\mu}\leq 2\epsilon\abs{J},\quad\text{$\forall$ large $k\in\N$.}
  \end{equation}
  By summing this over all the $J$'s in the partition, we deduce
  \eqref{eq:85b}.
\end{proof}

\begin{proof}[Proof of Theorem~\ref{thm:main}]
  Let $\theta_k(s)=\zeta(s)u(\phi_k(s))$ and $\psi_k(s)=u(\phi_k(s))$
  for all $s\in I$. Then by Theorem~\ref{thm:G-inv}, the \eqref{eq:85}
  holds for $\mu=\mu_G$. Therefore by Proposition~\ref{prop:P-}, we
  have \eqref{eq:85b}, which is same as \eqref{eq:4}.
\end{proof}

\begin{proof}[Proof of Theorem~\ref{thm:curve}]
  Due to regularity of Lebesgue measure, it is enough to prove the
  theorem under the assumption that
  \begin{equation}
    \label{eq:64}
    (\cI\circ\psi)^{(1)}\neq 0,\quad\forall s\in I.
  \end{equation}

  The map $\cS:\R^{n-1}\to \Sn^{n-1}$ defined by $\cS(x)=\cI(u(x))$ is
  the inverse stereographic projection. Therefore without loss of
  generality, we may assume that there exists a sequence
  $\phi_k\to\phi$ in $C(I,\R^{n-1})$ such that
  $\cI(\psi_k(s))=\cI(u(\phi_k(s))$ and $\cI(\psi(s))=\cI(u(\phi(s))$
  for all $s\in I$. Then by \eqref{eq:64} and \eqref{eq:90},
  \[
  \phi^{(1)}(s)\neq 0,\quad\text{and}\quad \abs{\{s\in
    I:\phi(s)\not\in \cS\inv(S)\}}=0, \quad\forall S\in\sS.
  \]
  Therefore by Remark~\ref{rem:sS}, the conclusion of
  Theorem~\ref{thm:G-inv} holds in the case of $x_k\to
  x_0=e\Gamma$. Therefore, since
  $P^-\psi_k(s)=P^-u(\zeta(s)\phi_k(s))$ for all $s\in I$,
  \eqref{eq:89} follows from Proposition~\ref{prop:P-}.
\end{proof}

\begin{proof}[Proof of Theorem~\ref{thm:uniform}]
  If the result fails to hold then there exist $f\in\Cc(G/\Gamma)$,
  $\epsilon>0$, a sequence $x_k\to x$ in $G/\Gamma$, a sequence
  $\{\psi_k\}$ of functions from $I\to G$ such that
  $\cI\circ\psi_k\mapsto \psi$ in $C^n(I,\Sn^{n-1})$, and an unbounded
  sequence $g_k\to\infty$ such that
  \begin{equation}
    \label{eq:87}
    \Abs{\frac{1}{\abs{I}}\int_If(g_k\psi_k(s)x_k)\,ds-\int_I
      f\,d\mu_G}\geq \epsilon.
  \end{equation}
  Since $G=KA^+K$, by passing to a subsequence, for each $k\in\N$ we
  have $g_k=h_k'a_{t_k}h_k$, where $h_k\to h$ and $h_k'\to h'$ in $K$
  as $k\to\infty$, and $t_k\to\infty$ in $\R$. Let $\tilde x\in G$ and
  $\tilde x_k\in G$ be such that $x_k=\tilde x_k\Gamma$, $x=\tilde
  x\Gamma$ and $\tilde x_k\to \tilde x$ as $k\to\infty$.

  Let $\bar\psi(s)=h\psi(s)\tilde x$ and
  $\bar\psi_k(s)=h_k\psi_k(s)\tilde x_k$ for all $s\in I$ and
  $k\in\N$. Then the condition of Theorem~\ref{thm:uniform} is
  satisfied for $\bar\psi$ in place of $\psi$ and $\bar\psi_k$ in
  place of $\psi_k$; note that we have used a stronger condition on
  $\psi$ that that \eqref{eq:94} holds for all proper subspheres $S$
  of $\Sn^{n-1}$ and $h\in G$. Therefore
  \begin{equation}
    \label{eq:93}
    \lim_{k\to\infty}\frac{1}{\abs{I}}\int_I
    f(h'a_{t_k}\bar\psi_k(s)\Gamma)\,ds=\int_{\gmg}
    f(h'y)\,d\mu_G(y)=\int_{\gmg} f\,d\mu_G. 
  \end{equation}
  Since $h_k'\to h'$ and $f$ is uniformly continuous, this equality
  contradicts \eqref{eq:87}.
\end{proof}

\begin{proof}[Proof of Theorem~\ref{thm:M}]
  As in the proof of Theorem~\ref{thm:uniform}, we need to show that
  given sequences $\psi_k\xrightarrow{k\to\infty} \psi$ in
  $C^n(I,T^1(M))$ and $t_k\xrightarrow{k\to\infty} \infty$ in $\R$,
  \begin{equation}
    \label{eq:91}
    \lim_{k\to\infty}\frac{1}{\abs{I}}\int_I
    f(g_{t_k}\psi_k(s))\,ds=\int_{T^1(M)} f\,d\mu, 
    \quad\forall f\in\Cc(T^1(M)). 
  \end{equation}  
  We will deduce this statement from Theorem~\ref{thm:curve}.
  
  There exists a lattice $\Gamma$ in $G=\SO(n,1)$ such that
  $T^1(M)\cong \SO(n-1)\backslash G/\Gamma$ and $T^1(\HH^n)\cong
  \SO(n-1)\backslash G$. Moreover the geodesic flow $\{g_t\}$ on
  $T^1(M)$ corresponds to the translation action of $\{a_t\}$ on
  $\SO(n-1)\backslash G/\Gamma$ from the left; the action is well
  defined because $\SO(n-1)\subset\cen(\{a_t\})$. Now the maps
  $\Vis:T^1(\HH^n)\to \Sn^{n-1}$ and $\bar\cI:\SO(n-1)\backslash G\to
  \Sn^{n-1}$ are same under the above identifications. Also the sets
  $\sS$ defined in \eqref{eq:95} and \eqref{eq:88}, as subsets of
  $\partial \HH^n$ and $P^-\backslash G$ respectively, are same under
  the above identification.

  The convergent sequence $\psi_k\to\psi$ in $C(I,T^1(M))$ can be
  lifted to a convergent sequence $\tilde \psi_k\to \tilde\psi$ in
  $C(I,T^1(\HH^n))$. Via the above correspondence, we obtain a
  convergent sequence $\tilde{\tilde{\psi}}_k\to\tilde{\tilde{\psi}}$
  in $C(I,G)$ such that $\Vis(\tilde
  \psi_k(s))=\cI(\tilde{\tilde{\psi}}_k(s))$ and
  $\Vis(\tilde\psi(s))=\cI(\tilde{\tilde{\psi}}(s))$. Therefore the
  conditions {a)} and {b)} on $\psi$ imply the condition \eqref{eq:90}
  of Theorem~\ref{thm:curve} for the map $\tilde{\tilde{\psi}}$. Also
  the required convergence property \eqref{eq:98} is satisfied for
  $\{\tilde{\tilde{\psi}}_k\}$. Now any $f\in\Cc(T^1(M))$ can be
  treated as a $\SO(n-1)$-invariant function on $G/\Gamma$. In this
  case, the conclusion \eqref{eq:89} of Theorem~\ref{thm:curve}
  holds. Therefore \eqref{eq:91} follows.
\end{proof}

The Theorem~\ref{thm:basic} is a special case of Theorem~\ref{thm:M}.

\section{Action of $\{a_t\}$ on shrinking curves}

Let $S\in \sS$ or $S=\Sn^{n-1}$. Define
\begin{equation}
  \label{eq:7}
  S^\ast= S\setminus \bigcup_{\substack{S'\subset S,\ \dim S'<\dim S\\S'\in \sS}} S'.
\end{equation}
Let $\phi\in C^n(I,\R^{d-1})$ and $g_0\in G$. We define
\begin{equation}
  \label{eq:22}
  I(S)=\{s\in I: \phi(s)\in \cS\inv(S^\ast g_0\inv)\}.  
\end{equation}
By the Lebesgue density theorem, almost every $x\in I(S)$ is a {\em
  density point\/} of $I(S)$; that is, if $I_k$ is any sequence of
intervals in $I$ containing $x$ such that $\abs{I_k}\to 0$, then
$\abs{I(S)\cap I_k}/\abs{I_k}\to 1$ as $k\to\infty$.

\begin{theo}
  \label{thm:short-generic}
  Let $x\in I$ such that $\phi^{(1)}(x)\neq 0$ and that $x$ is a
  density point for $I(S)$, where $S=\Sn^{n-1}$. Then for any
  sequences $\phi_k\to\phi$ in $C^n(I,\R^{n-1})$, $g_k\to g_0$ in $G$,
  $t_k\to\infty$ in $\R$ and any sequence of intervals
  $I_k\subset[a,b]$ such that $x\in I_k$, $\abs{I_k}\to 0$, and
  $\abs{I_k}^ne^{t_k}\to\infty$ the following holds: For any
  $f\in\Cc(G/\Gamma)$,
  \begin{equation}
    \lim_{k\to\infty}\frac{1}{\abs{I_k}}
    \int_{I_k}f(a_{t_k}u(\phi_k(s))g_0)\,ds = \int_{G/\Gamma} f\,d\mu_G.
  \end{equation}
\end{theo}

Let $\pi:G\to G/\Gamma$ denote the natural quotient map. Let $S\in\sS$
and $m=1+\dim S$. Let $g\in G$ be such that $S=\cI(\SO(m,1)g)$ and
$\nor(\SO(m,1))\pi(g)$ is closed. Due to the following claim, the coset
$\nor(\SO(m,1))g$ is uniquely defined.

We claim that if $F=\{h\in G: \cI(\SO(m,1)h)=\cI(\SO(m,1))\}$ then
$F=\nor(\SO(m,1)$. To prove the claim, we note that
$\nor(\SO(m,1))\subset F$. In particular $F$ is a reductive
group. Since $\nor(SO(m,1))$ is a symmetric subgroup of $G$, by
\cite[Cor.~4.7]{GOS:Satake}, $\nor(SO(m,1))$ is a maximal reductive
subgroup of $G$. Therefore $F=\nor(SO(m,1))$.

Let $L$ be the subgroup of $\nor(\SO(m,1))$ such that
$L\pi(g)=\cl{\SO(m,1)\pi(g)}$. Let $\mu_{L}$ denote the unique
$L$-invariant probability measure on $L\pi(g)$.

\begin{theo}
  \label{thm:short}
  Let $x\in I$ be such that $\phi^{(1)}(x)\neq 0$ and $x$ is a density
  point for the set $I(S)$. Then for any sequence $t_k\to\infty$ in
  $\R$, and any sequence of intervals $I_k\subset[a,b]$ such that
  $x\in I_k$, $\abs{I_k}\to 0$, and $\abs{I_k}^ne^{t_k}\to\infty$ the
  following holds: For any $f\in\Cc(G/\Gamma)$,
  \begin{equation}
    \lim_{k\to\infty}\frac{1}{\abs{I_k}}
    \int_{I_k}f(a_{t_k}u(\phi(s))g_0)\,ds = \int_{L\pi(g)} f(zy) \,d\mu_L(y),
  \end{equation}
  where $z\in \cen(A)\cap\SO(n)$ is such that $u(\phi(x))g_0\in
  U^-zLg$ and $u(\phi^{(1)}(x))\in zLz\inv$; $z$ depends only on
  $\phi$, $x$, $g_0$ and $Lg$.
\end{theo}

\begin{proof}[Proofs of Theorem~\ref{thm:short-generic} and
  Theorem~\ref{thm:short}]
 Let $x_0=\pi(g_0)$. For $k\in\N$, let $x_k=\pi(g_k)$ and $\lambda_k$
  be a probability measure on $G/\Gamma$ such that for any
  $f\in\cC(G/\Gamma)$,
  \[
  \int_{G/\Gamma} f\,d\lambda_k=\frac{1}{\abs{I_k}}\int_{I_k}
  f(a_{t_k}u(\phi_k(s)x_k)\,ds.
  \]
  
  First we note that Proposition~\ref{prop:CD} is valid for $J\subset
  I_k$.
  
  Therefore the proof of Theorem~\ref{thm:return} is valid in this
  case, and we obtain that after passing to a subsequence
  $\lambda_k\to\lambda$ in the space of probability measure on
  $G/\Gamma$.

  Let $W=\{u(r\phi^{(1)}(s)):r\in\R\}$. As in
  Theorem~\ref{thm:invariant}, we shall show that $\lambda$ is
  $W$-invariant. We will use the notation $\eta_1\simm{\epsilon}
  \eta_2$ to say that $\abs{\eta_1-\eta_2}\leq \epsilon$.

  Let $r\in\R$, $\epsilon>0$ and $f\in\Cc(G/\Gamma)$ be given. Due to
  uniform continuity of $f$ and equi-continuity of the family
  $\{\phi_k^{(1)}(s)\}$, and since $\abs{I_k}\to 0$, for sufficiently
  large $k\in\N$ and any $s\in I_k$, the following holds:
  \begin{equation}
    \label{eq:71}
    \begin{split}
      f(u(r\phi^{(1)}(x))a_{t_k}u(\phi_k(s))x_k)
      \simm{\epsilon}& f(u(r\phi_k^{(1)}(s))a_{t_k}u(\phi_k(s))x_k)\\
      =& f(a_{t_k}u(\phi_k(s)+e^{-t_k}r\phi^{(1)}(s))x_k)\\
      =& f(a_{t_k}u(\phi_k(s+re^{-t_k})+O(e^{-2t_k}))x_k)\\
      =& f(u(O(e^{-t_k}))a_{t_k}u(\phi_k(s+re^{-t_k}))x_k)\\
      \simm{\epsilon}& f(a_{t_k}u(\phi_k(s+re^{-t_k}))x_k).
    \end{split}
  \end{equation}
  Therefore, for sufficiently large $k\in\N$,
  \begin{equation}
    \label{eq:103}
    \begin{split}
     \int f(u(r\phi^{(1)}(x))y)\,d\lambda_k(y)
      \simm{\epsilon}
      &\frac{1}{\abs{I_k}}\int_{I_k}f(a_{t_k}u(\phi_k(s+re^{-t_k}))x_k)\,ds\\
      \simm{\epsilon}
      &\frac{1}{\abs{I_k}}\int_{I_k}f(a_{t_k}u(\phi_k(s))x_k)\,ds
      =\int f\,d\lambda_k,
    \end{split}
  \end{equation}
  where the last approximation holds because
  \[
  \abs{I_k}^ne^{t_k}\to\infty \implies
  2\frac{1}{\abs{I_k}}\abs{re^{-t_k}}\sup \abs{f}\to 0.
  \]
  From \eqref{eq:103} we deduce that $\lambda$ is $W$-invariant.

  The proof of Proposition~\ref{prop:cusps} goes through in this
  case. We can now apply Ratner's classification of ergodic invariant
  measures exactly as in the earlier case. Then we follow the the
  Proof of Theorem~\ref{thm:G-inv} for $I_k$ in place of $I$. We have
  $E(k)=E_1(k)\cup E_2(k)$. The same proof goes through to say that
  for sufficiently large $k$, $\abs{E_2(k)}\leq
  \epsilon\abs{I_k}$. The basic difference occurs in analyzing
  $\abs{E_1(k)}$. Again all the arguments are valid up to
  \eqref{eq:80} for $I_k$ in place of $I$; we get
  \begin{equation}
    \label{eq:80s}
    \abs{\{s\in I_k: \cS(\phi(s))\in S'g_0\inv\}}
    \geq \epsilon\abs{I_k}/\card{\Sigma},
  \end{equation}
  for some $S'\in \sS$. By our hypothesis, $x$ is a density point of
  $I(S)$ (see \eqref{eq:22}). Therefore in view of the definition of
  $S^\ast$, we deduce that $S\subset S'$.

  In particular, if $S=\Sn^{n-1}$ then this is not possible. Therefore
  as in the Proof of Theorem~\ref{thm:G-inv} we conclude that
  \eqref{eq:79} fails to hold, and in turn \eqref{eq:61} fails to
  hold, and hence $\lambda$ is $G$-invariant. Thus the proof of
  Theorem~\ref{thm:short-generic} is complete.

  Now for Theorem~\ref{thm:short} we have that $\phi_k=\phi$ and
  $g_k=g_0$ for all $k\in\N$. For any $s\in I(S)$, there exists
  $b(s)\in M$ such that $b(s)\to z$ as $s\to x$ and
  \[
  u(\phi(s))g_0\in U^-b(s)Lg.
  \]
  Therefore, since $x$ is a density point of $I(S)$ and $\lambda_k\to
  \lambda$, from the definition of $\lambda_k$ we conclude that
  $\supp\lambda\subset zL\pi(g)$. Since $\lambda(\pi(N(H,W))>0$, we
  conclude that $S'g_0\inv\subset Sg_0\inv$. Thus $S=S'$ and
  $H^\nc\cong \SO(m,1)$, and $\supp(\lambda)\subset
  g'\nor(H^\nc)\pi(e)$ for some $g'\in G$ such that $AW\subset
  g'\nor(H^\nc)(g')\inv$. Since $\lambda(\pi(S(H,W)))=0$, we deduce
  that each $W$-ergodic component of $\lambda$ is invariant under
  $g'H^\nc(g')\inv$. Therefore, since 
\[
\supp(\lambda)\subset zL\pi(g) \quad\text{and}\quad
L\pi(g)=\cl{\SO(m,1)\pi(g)},
\]
by dimension consideration, we conclude that $\supp(\lambda)=zL\pi(g)$
and $\lambda=z\mu_L$. This completes the proof of
Theorem~\ref{thm:short}.
\end{proof}

\section{Evolution of shrinking curves under geodesic flow}
\label{sec:short-geom}

Let the notation be as in \S\ref{subsec:geom}. As a consequence of
Theorem~\ref{thm:short-generic} we obtain the following.

\begin{theo}
  \label{thm:generic-geom}
  Let $\psi\in C^n(I,T^1(M))$. Let $x\in I$ be such that
  $(\Vis\circ\tilde\psi)^{(1)}(x)\neq 0$ and that $x$ is a density
  point of $I(M)$. Then for any sequence $\psi_k\to\psi$ in
  $C^n(I,T^1(M))$, a sequence $t_k\to\infty$, and a sequence $I_k$ of
  subintervals of $I$ containing $x$ such that $\abs{I_k}\to 0$ and
  $\abs{I_k}^ne^{t_k}\to\infty$ the following holds:
  \begin{equation}
    \label{eq:104}
    \lim_{k\to\infty} \frac{1}{\abs{I_k}}\int_{I_k}
    f(g_{t_k}\psi_k(s))\,ds=
    \int_{T^1(M)} f\,d\mu_M, \quad \forall f\in \Cc(T^1(M)).
  \end{equation}
\end{theo}

As a consequence of Theorem~\ref{thm:short} we deduce the following:

\begin{theo}
  \label{thm:short-geom}
  Let $\psi\in C^n(I,T^1(M))$. Let $M_1\in \bar\sS$ and $S\in
  \sS(M_1)$.  Let $x\in I$ be such that
  $(\Vis\circ\tilde\psi)^{(1)}(x)\neq 0$ and $x$ is a density point of
  $I(S)$. Then given any sequence $t_k\to\infty$ in $\R$ and a
  sequence of subintervals $I_k$ of $I$ containing $x$ such that
  $\abs{I_k}\to0$ and $\abs{I_k}^ne^{t_k}\to\infty$ the following
  holds:
  \begin{equation}
    \label{eq:107}
    \lim_{k\to\infty}  \frac{1}{\abs{I_k}}\int_{I_k} f(a_{t_k}\psi(s))\,ds
    =\int_{T^1(M_1)} f\,d\mu_{M_1}, \quad \forall f\in\Cc(T^1(M)).
  \end{equation}
\end{theo}

The conclusion of Theorem~\ref{thm:smooth-general} can be deduced from
Theorem~\ref{thm:short-geom} using using regularity Lebesgue measure
and standard arguments of measure theory.

\subsection{Geodesic evolution of faster shrinking curve}

We can also obtain following variations of
Theorem~\ref{thm:generic-geom} and Theorem~\ref{thm:short-geom}. 

\begin{theo}
  \label{thm:generic-geom-fast}
  In the statement of Theorem~\ref{thm:generic-geom} suppose that
  $\psi_k\to \psi$ in $C^{2n-2}(I,T^1(M))$. Then given a sequence
  $t_k\to\infty$, and a sequence of subintervals $I_k$ of $I$
  containing $x$ such that $\abs{I_k}\to0$ and
  $\abs{I_k}^2e^{t_k}\to\infty$, the equation \eqref{eq:104} holds.
\end{theo}

\begin{theo}
  \label{thm:short-fast}
  In the statement of Theorem~\ref{thm:short-geom} suppose that
  $\psi\in C^{2n-2}(I,T^1(M))$. Then given a sequence $t_k\to\infty$, and a
  sequence of subintervals $I_k$ of $I$ containing $x$ such that
  $\abs{I_k}\to 0$ and $\abs{I_k}^2e^{t_k}\to\infty$, the
  equation~\eqref{eq:107} holds.
\end{theo}

It is interesting to compare these statements with the results in
\cite{Strom:long-horo}.

To prove the above theorems using the method of this article, the
only property required to be verified is the following variation of
Proposition~\ref{prop:main}.

\begin{prop}[Basic Lemma-II]
  \label{prop:main-variation}
  Let $\phi_k\to\phi$ in $C^{2n-2}(I,\R^{d-1})$.  Given $C>0$, there
  exists $R_0>0$ such that for any sequence $t_k\to\infty$ in $\R$
  there exists $k_0\in\N$ such that for any $x\in I=[a,b]$ and $v\in
  V$, there exists an interval $[s_k,s'_k]\subset I$ containing $x$
  such that for any $k\geq k_0$, the following conditions are
  satisfied:
  \begin{align}
    e^{t_k}(s'_k-s_k)^2&<C,\\
    \norm{a_{t_k}u(\phi_k(s_k))v}&\geq \norm{v}/R_0, \quad \text{if
      $s_k>a$,}\\
    \norm{a_{t_k}u(\phi_k(s'_k))v}&\geq \norm{v}/R_0, \quad \text{if
      $s'_k<b$.}
  \end{align}
\end{prop}

\begin{proof}
  We follow the strategy of the proof of
  Proposition~\ref{prop:main}. We will now highlight some crucial
  modification required in the proof.

  First \eqref{eq:3b} is replaced by $e^{t_k}\delta_k^2\geq C$. We put
  $\mu=n-1$ in \eqref{eq:13} to get
  \[
  \sup_{s\in [0,\delta_k]}\norm{q_{n-1}(u(\phi_{k,r_k}(s))w_k)}\leq
  R_k\inv C\inv \delta_k^{2(n-1)}.
  \]
  Therefore in place of \eqref{eq:27} we will have
  \begin{equation}
    \label{eq:27-2}
    \lim_{s\to 0}\norm{q_{n-1}(u(\phi_{0,r_0}(s))w_0)}/s^{2n-2}=0.
  \end{equation}
  Now following the further arguments using the
  $\SL(2,\R)$-representation theory, we will obtain an analogue of
  \eqref{eq:23} for $q_{n-1}$ involving $h^{(n-1)-\mu_0}$ in the
  highest order term. Therefore \eqref{eq:26} will become
  \[
  \lim_{s\to 0}
  \norm{q_{n-1}(u(\phi_{0,r_0}(s))w_0)}/s^{(n-1)-\mu_0}\geq\eta_0\rho_0^{n-1-\mu_0}>0.
  \]
  Since $n-1-\mu_0\leq 2n-2$, this will contradict \eqref{eq:27-2}.
\end{proof}

\begin{rem} 
  \label{rem:linear} Using Proposition~\ref{prop:main-variation}, we
  can obtain an analogue of Corollary~\ref{cor:phi-poly} for
  $P_{k,x}(s)=\phi_k(x)+s\phi_k^{(1)}(s)$. Thus for linearization
  technique, we can approximate a $C^{2(n-1)}$-curve $\phi_k$ at any
  $s\in I$ by its tangent line, rather than a polynomial curve.
\end{rem}

\end{document}